\documentclass[a4paper,10pt]{amsart}
\usepackage[utf8]{inputenc}
\usepackage{amsthm}
\usepackage{amsmath}
\usepackage{bm}
\usepackage{amsfonts}
\usepackage{amssymb}
\usepackage{mathtools}
\usepackage{mathrsfs}
\usepackage{setspace}
\usepackage{xcolor}
\usepackage{textgreek}
\usepackage{todonotes}
\usepackage{thmtools} % solves problem of shared counters in statements and autoref

\usepackage[pdfdisplaydoctitle,colorlinks,breaklinks,urlcolor=blue,linkcolor=blue,citecolor=blue]{hyperref} %must always be the last

\newcommand{\C}{\mathbb{C}}

\newcommand{\F}{\mathcal{F}}

\renewcommand{\H}{\mathcal{H}}

\newcommand{\N}{\mathbb{N}}

\newcommand{\PP}{\mathbb{P}}

\newcommand{\R}{\mathbb{R}}

\newcommand{\Z}{\mathbb{Z}}

\DeclareMathOperator{\id}{Id}
\let\div\relax
\DeclareMathOperator{\div}{div}

\renewcommand{\epsilon}{\varepsilon}

\newcommand{\obar}[1]{\overline{#1}}

\DeclareMathOperator{\imm}{i}

\newcommand{\de}{\partial}

\newcommand{\set}[1]{\left\{#1\right\}}
\newcommand{\pa}[1]{\left(#1\right)}

\newcommand{\abs}[1]{\left|#1\right|}
\newcommand{\norm}[1]{\left\|#1\right\|}
\newcommand{\brak}[1]{\left\langle#1\right\rangle}

\newcommand{\expt}[1]{\mathbb{E}\left[#1\right]}
\def\Xint#1{\mathchoice
	{\XXint\displaystyle\textstyle{#1}}%
	{\XXint\textstyle\scriptstyle{#1}}%
	{\XXint\scriptstyle\scriptscriptstyle{#1}}%
	{\XXint\scriptscriptstyle\scriptscriptstyle{#1}}%
	\!\int}
\def\XXint#1#2#3{{\setbox0=\hbox{$#1{#2#3}{\int}$}
		\vcenter{\hbox{$#2#3$}}\kern-.5\wd0}}
\def\dashint{\Xint-}
\newtheorem{thm}{Theorem}[section]

\newtheorem{definition}[thm]{Definition}
\newtheorem{cor}[thm]{Corollary}

\newtheorem{lem}[thm]{Lemma}

\newtheorem{prop}[thm]{Proposition}

\theoremstyle{remark}
\newtheorem{rmk}[thm]{Remark}

\numberwithin{equation}{section}

\title[Equilibrium Statistical Mechanics Barotropic Quasi-Geostrophic]{Equilibrium Statistical Mechanics\\of Barotropic Quasi-Geostrophic Equations}
\author[F. Grotto]{Francesco Grotto}
  \address{Scuola Normale Superiore, Piazza dei Cavalieri, 7, 56126 Pisa, Italia}
  \email{\href{mailto:francesco.grotto@sns.it}{francesco.grotto@sns.it}}
\author[U. Pappalettera]{Umberto Pappalettera}
  \address{Scuola Normale Superiore, Piazza dei Cavalieri, 7, 56126 Pisa, Italia}
  \email{\href{mailto:umberto.pappalettera@sns.it}{umberto.pappalettera@sns.it}}
\date\today

\begin{document}

\begin{abstract}
 We consider equations describing a barotropic inviscid flow in a channel 
 with topography effects and beta-plane approximation
 of Coriolis force, in which a large-scale mean flow interacts with smaller scales.
 Gibbsian measures associated to the first integrals energy and enstrophy are Gaussian measures supported
 by distributional spaces. 
 We define a suitable weak formulation for barotropic equations, and prove 
 existence of a stationary solution preserving Gibbsian measures,
 thus providing a rigorous infinite-dimensional
 framework for the equilibrium statistical mechanics of the model.
\end{abstract}

\maketitle

%%%%%%%%%%%%%%%%%%%%%%%%%%%%%%%%%%%%%%%%%%%%%%%%%%%%%%%%%%%%%%%%%%%%%%%%%%%%%%%%%%%%%%%%%%%%%
\section{Introduction}\label{sec:introduction}

Barotropic quasi-geostrophic equations in channel domains constitute a physically
relevant partial differential equation in oceanography and atmospheric modeling,
with applications including for instance the Antarctic circumpolar current.
Significance of the model is discussed for instance in \cite{GrMu96,ChZMGh03,MaWa06,DiSeSheWa15} and references therein,
to which we refer.

The presence of conserved quantities and their associated equilibrium statistical mechanics constitute an
important feature of the model, and our work will focus on invariant measures and stationary solutions.
Although numerical reasons naturally lead to consider Fourier truncated or other approximations
of the stationary flow, as for instance in \cite[Section 6]{MaTiVE01}, the full infinite-dimensional
setting is of great interest because of its geophysical relevance and mathematical difficulty,
as discussed in \cite{MaWa06}. The latter monography thoroughly discusses in Chapter 8
equations for fluctuations around the mean state for the truncated model,
and then considers a continuum limit by scaling parameters of invariant measures so to
neglect fluctuations, obtaining a mean state description for the PDE model.

Our contribution in a sense furthers their study: we will show how fluctuations can be included
in the continuum limit by defining a suitably weak notion of solution,
so to include the distributional regimes dictated by the full infinite-dimensional invariant measure,
under which fluctuations of comparable order are observed at all scales.

The model under consideration, for the derivation of whom we refer to \cite[Chapter 1]{MaWa06}, is the following. 
We consider the rectangle $R=[-\pi,\pi]\times [0,\pi]$ as a space domain,
and denote $z=(x,y)\in R$ its points; we also fix a finite interval for time $t\in [0,T]$.
The governing dynamics is the inviscid quasi-geostrophic equation for the scalar \emph{potential vorticity}
$q(t,z)$,
\begin{equation}\label{eq:activescalar}
\de_t q+ \nabla^\perp \psi \cdot \nabla q=0,
\end{equation}
where $\nabla^\perp=(-\de_y,\de_x)$, and $\psi(t,z)$ is the \emph{stream function} determining
the divergence-less velocity field $\nabla^\perp\psi$.
The channel geometry prescribes that velocity $\nabla^\perp \psi$ be tangent to the top and bottom
boundaries of $R$, and we further assume the flow to be periodic in the $x$ coordinate.
Such boundary conditions are encoded in terms of $\psi$ as follows:
\begin{align}\label{eq:boundary1}
&\de_x\psi(t,x, \pi)=\de_x\psi(t,x, 0)=0,\\ \label{eq:boundary2}
&\nabla^\perp\psi(t,x+2 \pi, y)=\nabla^\perp\psi(t,x, y).
\end{align}
As a consequence, at fixed $t$ the stream function $\psi$ is constant on the impermeable boundaries $y=0,\pi$.
Using the fact that $\psi$ is defined up to an additive constant,
possibly depending on time, we will set $\psi(t,x,0)\equiv 0$,
from which it is easily seen that $\psi$ takes the form
\begin{equation*}
\psi=-V y+\psi',
\end{equation*}
with $V(t)$ a function of time only describing a large-scale mean flow,
and $\psi'(t,z)$ the scalar \emph{small-scale stream function}, periodic in $x$ and null at $y=0,\pi$ at all times.
Potential vorticity is then linked to $\psi'$ by
\begin{equation}\label{eq:poisson}
q=\Delta \psi'+h+\beta y,
\end{equation}
where $h(z)$ is a smooth scalar function modelling the effect of the underlying topography on the fluid,
and $\beta y$, $\beta\in\R$, is the beta-plane approximation of Coriolis' force.

Dynamics of $V(t)$ is derived by imposing conservation of \emph{total energy},
\begin{equation}
E=\frac12\dashint_R|\nabla^\perp\psi|^2dxdy=\frac12 V^2+\frac12\dashint_R|\nabla^\perp\psi'|^2dxdy,
\end{equation}
from which one obtains an equation for time evolution of the mean flow,
\begin{equation}\label{eq:bc}
\frac{d V}{dt}=-\dashint_R\de_x h(z) \psi'(z) dz,
\end{equation}
the right-hand side being usually referred to as \emph{topographic stress}. This last relation
completes our set of equations,
\begin{equation}\label{eq:bqg}\tag{BQG}
\begin{cases}
\de_t q+ \nabla^\perp \psi \cdot \nabla q=0,\\
q=\Delta \psi' +h+\beta y,\\
\psi=-V y+\psi',\\
\frac{d V}{dt}=-\dashint_R\de_x h(z) \psi'(z) dz.
\end{cases}
\end{equation}
Since both $\psi$ and $\psi'$ can be recovered from $V$ and $q$, taking into account the boundary conditions 
\eqref{eq:boundary1}, \eqref{eq:boundary2} in solving Poisson's equation \eqref{eq:poisson}, 
we will consider $(V,q)$ as the state variables of the system. 
This particular choice has the advantage of retaining the active scalar form
for the dynamics \eqref{eq:activescalar} of $q$.

Our study focuses on equilibrium statistical mechanics of \eqref{eq:bqg} in the full infinite dimensional setting,
generalising the well-established theory developed for 2-dimensional Euler equations.
Besides the total energy $E$, \eqref{eq:bqg} preserve the \emph{large-scale enstrophy}
\begin{equation}
Q(V,q)=\beta V+\frac12 \dashint_R (q-\beta y)^2.
\end{equation}
Due to the Hamiltonian nature of the fluid dynamics, it is thus expected that the Gibbsian ensembles
\begin{equation}\label{eq:gibbsianmeasures}
d\nu_{\alpha,\mu}(V,q)=\frac1{Z_{\alpha,\mu}} e^{-\alpha(\mu E(V,q)+Q(V,q))}dVdq, \quad \alpha,\mu>0,
\end{equation}
are invariant measures for \eqref{eq:bqg}.
Since Boltzmann's exponents are quadratic functionals of the state variables $(V,q)$, these
are Gaussian measures. Unfortunately, they are only supported on spaces of distributions
--they give null mass to any space of functions-- so some effort is required to give
meaning to the dynamics \eqref{eq:bqg} in the low-regularity regime dictated by $\nu_{\alpha,\mu}$.

Inspired by works on Euler's equations, \cite{AlCr90,Fl18}, we describe a weak formulation of
\eqref{eq:bqg} robust enough to admit samples of $\nu_{\alpha,\mu}$ as fixed-time distributions,
and then produce by means of a Galerkin approximation scheme such a solution,
arriving at our main result:

\begin{thm}\label{thm:maintheorem}
	Let $\beta\neq 0$ and $h$ as above.
	For any $\alpha,\mu>0$ there exists a stationary stochastic process 
	$(V_t,q_t)_{t\in [0,T]}$ with fixed-time marginals $\nu_{\alpha,\mu}$,
	whose trajectories solve \eqref{eq:bqg} in the weak vorticity formulation of \autoref{def:bqgvorticity}.
\end{thm}

The reader should be aware that, although we often make use of terminology from Probability theory,
there is no noise or external randomness acting on the system under consideration,
and we are only referring to the fact that an invariant measure of a deterministic evolution
can be regarded as a random initial data producing a stationary (deterministic) process.

As in the case of 2-dimensional Euler's equations in the Energy-Enstrophy stationary regime,
or more generally when fixed time marginals are absolutely continuous with respect to space white noise, see \cite{FlGrLu19},
uniqueness remains an important open problem. We will not discuss uniqueness of solutions of \eqref{eq:bqg}
in the above stationary regime; thus, in particular, we are not able to state that the solutions
we produce form a flow, \emph{i.e.} a one-parameter group of transformations of phase space indexed by time.

The present paper is structured as follows. In \autoref{sec:definitions} we collect some
preliminary material, including a short discussion on regularity regimes in which \eqref{eq:bqg} are well-posed.
In \autoref{sec:weaksolutions} we thoroughly discuss the formulation of weak solution
required by our low-regularity setting, and finally in \autoref{sec:galerkin} we will prove
\autoref{thm:maintheorem} by approximating the infinite-dimensional stationary solution
with finite-dimensional, stationary Galerkin truncations of \eqref{eq:bqg}.

%%%%%%%%%%%%%%%%%%%%%%%%%%%%%%%%%%%%%%%%%%%%%%%%%%%%%%%%%%%%%%%%%%%%%%%%%%%%%%%%%%%%%%%%%%%%%
\section{Definitions and Preliminary Results}\label{sec:definitions}

We consider mixed boundary conditions on $R$ for the small scale stream function $\psi'$,
that is periodicity in the $x$ variable and Dirichlet boundary at $y=0,\pi$.
In order to simplify Fourier analysis, let us extend the space domain to the 2-dimensional torus
$D=[-\pi,\pi]^2$ with periodic boundary conditions on both $x,y$ variables, extending $\psi'$
to $D$ so that it becomes an odd function of $y$. We still denote points $z=(x,y)\in D$.

To study \eqref{eq:bqg} on $D$ we also extend $q,h$ in the same way; the extension of $h$
might be discontinuous at $y=0$, but this will not be relevant in the following.
Indeed, it is not difficult to see that equations \eqref{eq:bqg} preserve such condition.
We also remark that the beta-plane term $\beta y$ of \eqref{eq:poisson} is coherent with the
domain extension.

Due to the (skew-)symmetry in $y$ variable, it will be convenient to introduce 
the following set of orthonormal functions of $L^2(D,\C)$,
\begin{gather*}
(e_j)_{j \in \Z}, \pa{e_j s_k , e_j c_k}_{(j,k) \in \Lambda} \quad  \Lambda = \left\{ (j,k) : j\in \Z, k \in \N\setminus \{0\},  \right\}\\
e_j(x) = \frac{1}{2\pi}e^{\imm jx}, \,\, s_k(y) = \sin(ky), \,\, c_k(y) = \cos(ky).
\end{gather*}
Since we work with real valued objects, Fourier coefficients relative to modes 
$(j,k)$ and $(-j,k)$ will always be complex conjugated. 
With this relation between Fourier coefficients, $\left\{e_j, e_j s_k , e_j c_k \right\}_{(j,k) \in \Lambda}$ 
is a Hilbert basis of $L^2=L^2(D,\R)$.

Odd functions of $y$ only have non null Fourier coefficients relative to $\pa{e_js_k}_{(j,k) \in \Lambda}$:
we will denote those coefficients, say of $\psi'$, by
\begin{equation*}
\F_{j,k}(\psi')=\hat{\psi}'_{j,k}=\int_D \psi'(x,y) e_{-j}(x)s_k(y)dxdy.
\end{equation*}
so that
\begin{equation*}
\psi'(x,y) = \sum_{(j,k) \in \Lambda} \hat{\psi}'_{j,k} e_j(x)s_k(y), 
\quad \hat{\psi}'_{j,k} = \overline{\hat{\psi}'_{-j,k}}.
\end{equation*} 

For $\alpha\in\R$, we denote by $H^\alpha=W^{\alpha,2}(D,\R)$ the $L^2(D,\R)$-based Sobolev spaces, 
which enjoy the compact embeddings $H^\alpha\hookrightarrow H^\beta$ whenever $\beta<\alpha$,
the injections being furthermore Hilbert-Schmidt if $\alpha>\beta+1$.
The scale of Sobolev spaces of odd distributions in $y$,
\begin{align*}
\H^\alpha = \set{u=\sum_{(j,k) \in \Lambda} \hat u_{j,k} e_js_k:
	\norm{u}^2_{\H^\alpha}=\sum_{(j,k) \in \Lambda} \abs{\hat u_{j,k}}^2 (j^2+k^2)^{2\alpha}<\infty},
\end{align*}
clearly share the same properties. We denote with $\H^0$ the subspace of odd functions of $y$ in $L^2(D)$,
and more generally each $\H^\alpha$ is a closed subspace of $H^\alpha$. Brackets $\brak{\cdot,\cdot}$ will denote $\H^0$-based duality couplings. 

As a convention, $C$ will denote a positive constant, possibly changing in every occurrence even in the same formula
and depending only on its eventual subscripts.

\subsection{Well-posedness regimes}

Our main aim is to give meaning to \eqref{eq:bqg} in distributional regimes
dictated by the formally invariant Gibbs measures. Before we undertake that task, 
we briefly discuss, for the sake of completeness, more regular regimes in which 
our equations are actually well-posed.
Let us begin by introducing the notion of weak solution.

\begin{definition}\label{bqgweakdef}
	Given $(V_0,q_0)\in \R\times L^\infty(D)$, we say that
	\begin{equation*}
	(V(t),q(t))_{t\in [0,T]}\in L^\infty([0,T], \R\times D)
	\end{equation*}
	is a weak solution to \eqref{eq:bqg} with initial datum $(V_0,q_0)$ if for any 
	$\varphi\in C^1([0,T]\times D)$ it holds
	\begin{align}\label{eq:weakformulationq}
	&\int_D \varphi(T,z) q(T,z) dz-\int_D \varphi(0,z) q_{0}(z) dz\\ \nonumber
	&\qquad =\int_{0}^{T} \int_D (\partial_t\varphi(s,z)+\nabla^\perp\psi(s,z)\cdot \nabla\varphi(s,z)) q(s,z)dz d s,\\
	\label{eq:weakformulationV}
	&V(t)=V_0+\int_0^t \dashint_D h(z) \de_x\psi'(z,s) dz ds,\\
	\label{eq:weakformother}
	&\psi=-V y+\psi', \quad q=\Delta\psi'+h+\beta y.
	\end{align}
\end{definition}

Thanks to the fact that the equation for $q$ is in the active scalar form,
the method of characteristics produces an existence result:
a minor modification of the proof of \cite[Ch.2,Theorem 3.1]{MaPu94} leads to the following:

\begin{prop}\label{thm:regularexistence}
	Let $(V_0,q_0)\in \R\times L^\infty(D)$, and consider the Lagrangian formulation of \eqref{eq:bqg}
	given by
	\begin{gather}\label{eq:cauchyflow}
	\begin{cases}
	\frac{d}{dt}\phi_t(z)= \nabla^\perp \psi(t,\phi_t(z))\\
	\phi_0(z)=z
	\end{cases},\quad q(t,z)=q_0(\phi_{-t}(z)),
	\end{gather}
	together with equations \eqref{eq:weakformulationV},\eqref{eq:weakformother}.
	There exists a unique solution $(\phi,V,q)$ of such system, 
	and moreover $(V,q)$ is a weak solution of \eqref{eq:bqg} in the sense of \autoref{bqgweakdef}.
\end{prop}

The argument ultimately relies on the fact that $\nabla\nabla^\perp\Delta^{-1}$ 
is a singular kernel of Calder\'on-Zygmund type, so that its associated convolution operator is a bounded linear map from $L^\infty(D)$ to the Bounded Mean Oscillation (BMO) space.
This implies that the vector field
\begin{equation*}
\nabla^\perp \psi= V\binom{0}{1}+\nabla^\perp \Delta^{-1}(q-h-\beta y)
\end{equation*}
has gradient in BMO, and thus it is log-Lipschitz (\emph{cfr.} \cite[Ch.2,Lemma 3.1]{MaPu94}).
The vector field $\nabla^\perp \psi$ then satisfies the Osgood condition (\cite{Os98}) for the associated Cauchy problem \eqref{eq:cauchyflow}, which is thus well-posed; 
it is not difficult to check that $q(t,z)=q_0(\phi_{-t}(z))$ satisfies the weak formulation \eqref{eq:weakformulationq}.
All these ideas date back to the celebrated work of Judovi\v{c}, \cite{Ju63}, concerning well-posedness
of Euler equations for initial vorticity in $L^\infty$.

\begin{prop}\label{thm:regularuniqueness}
	For any $(V_0,q_0)\in \R\times L^\infty(D)$,
	the weak solution of \eqref{eq:bqg} in the sense of \autoref{bqgweakdef} is unique.
\end{prop}

Uniqueness can be obtained by energy estimates at the level of the velocity vector field $v=\nabla^\perp \psi$.
Such estimates are performed for instance in \cite[Theorem 8.2]{MaBe02} for the 2D Euler equations
($h=0$, $\beta=0$), and again they rely on the fact that $\nabla\nabla^\perp \psi$ is in BMO
to arrive at Gronwall-type inequalities,
something which is not influenced by the addition of regular terms such as $h+\beta y$ to $q$. 
We refer to \cite{AzBe15} for a thorough discussion of uniqueness for a large class of active scalar equations
sharing similar features.
We also mention the recent work \cite{Ch19a}, where the arguments we just sketched
are applied to a barotropic quasi-geostrophic model closely related to ours: the difference consists
in impermeable boundary conditions on the whole boundary and the presence of a free surface effect instead of the
fixed topography $h$ we consider. The paper \cite{Ch19b}, moreover, is devoted to multi-layered 
barotropic quasi-geostrophic equations.

\subsection{Conserved Quantities and Gibbsian Measures}\label{ssec:gibbsianmeasures}
Smooth solutions of \eqref{eq:bqg} preserve the first integrals energy and enstrophy,
\begin{equation*}
E=\frac12V^2+\frac12 \dashint \abs{\nabla^\perp\psi'}^2, \quad Q=\beta V+\frac12 \dashint (q-\beta y)^2.
\end{equation*}
We refer again to \cite[Section 1.4]{MaWa06} for a detailed discussion of conserved quantities.
As already remarked, energy $E$ can be seen as a functional of variables $(V,q)$
by solving the Poisson equation \eqref{eq:poisson}.

In \eqref{eq:gibbsianmeasures} above, we have formally introduced the Gibbsian measures
\begin{equation*}
d\nu_{\alpha,\mu}(V,q)=\frac1{Z_{\alpha,\mu}} e^{-\alpha(\mu E+Q)}dVdq, \quad \alpha,\mu>0,
\end{equation*}
the expression meaning that we consider the Gaussian measure whose inverse covariance operator
is given by the quadratic functional $\alpha(\mu E+Q)$ of $(V,q)$.

Let us now provide a rigorous framework: we define $\nu_{\alpha,\mu}$ as the joint law of
the Gaussian variable $V\sim N\pa{-\frac\beta\mu,\frac1{\alpha\mu}}$ and the Gaussian random field $q$
indexed by $\H^0$ with mean and covariance given by, for $f,g\in \H^0$,
\begin{align*}
\expt{\brak{q,f}}=\brak{\bar q,f}&=\brak{\frac\mu{\mu-\Delta} h+\beta y,f}\\
\expt{\brak{q,f}\brak{q,g}}-\brak{\bar q,f}\brak{\bar q,g}
&=\brak{f, \frac1{\alpha(1-\mu\frac1\Delta)}g},
\end{align*}
$V$ and $q$ being independent. Notice that $\alpha$ only plays a role in the variance. 
The link between the latter and the formal definition \eqref{eq:gibbsianmeasures}
is perhaps clearer thinking of the formal reference measure $dVdq$ as the infinite product of uniform
measures on the infinite product space $\R\times \C^\Lambda$ of Fourier modes (modulo the relation $\hat{q}_{j,k} = \overline{\hat{q}_{-j,k}}$),
and considering the Boltzmann exponent $e^{-\alpha(\mu E+Q)}$ as the infinite product of
densities given by the Parseval expansion of the quadratic form $\alpha(\mu E+Q)$.  

In order to deal with centred variables we set
\begin{equation}\label{eq:centering}
U=V+\frac\beta\mu, \quad \omega=q-\bar q,
\end{equation}
the new variables satisfying equations of motion
\begin{equation} \label{eq:bqgreduced}
\begin{cases}
\de_t\omega+\nabla^\perp \Delta^{-1}\omega\cdot \nabla\omega+ L\omega=0\\
\frac{d U}{dt}=\dashint_D h  \de_x \Delta^{-1}\omega
\end{cases},
\end{equation}
where $L\omega$ collects all affine terms in $\omega$,
\begin{align*}
L\omega&=\pa{U-\frac\beta\mu}\de_x\omega
+U \frac{\mu\de_x}{\mu-\Delta}h
+\frac{\nabla^\perp}{\mu-\Delta}h\cdot\nabla\omega\\
&\qquad +\nabla^\perp\Delta^{-1}\omega \cdot\frac{\mu\nabla}{\mu-\Delta}h+\beta \de_x\Delta^{-1}\omega.
\end{align*}
The equivalence of \eqref{eq:bqgreduced} and \eqref{eq:bqg} is intended for smooth solutions. 

We now define the purely quadratic \emph{pseudoenergy}: for $\mu>0$,
\begin{equation}\label{eq:pseudoenergy}
S_\mu(U,\omega)=\frac\mu{2}U^2+\frac12 \int_{D} (\omega-\mu \Delta^{-1}\omega)\omega dxdy
\end{equation}
so that the law of $(V,\omega)$ under $\nu_{\alpha,\mu}$ is given by
\begin{equation} \label{eq:gibbs}
d\eta_{\alpha,\mu}(U,\omega)=\frac1{\tilde Z_{\alpha,\mu}} e^{-\alpha S_\mu(U,\omega)}dUd\omega,
\end{equation}
the latter to be interpreted analogously to the definition of $\nu_{\alpha,\mu}$ above, \eqref{eq:gibbsianmeasures}:
it is the joint law of the real Gaussian variable $U\sim N\pa{0,\frac1{\alpha\mu}}$ and the independent centred
Gaussian field $\omega$ with covariance operator $\alpha^{-1}(1-\mu\Delta^{-1})^{-1}$.
In order to lighten the exposition, we will abuse notation denoting by $\eta_{\alpha,\mu}(d\omega)$ the law of
$\omega$ under $\eta_{\alpha,\mu}$, and analogously for $U$.
We will also denote
\begin{equation*}\label{eq:variance}
\sigma^2\coloneqq\int U^2 d\nu_{\alpha,\mu}(U,\omega)=\frac1{\alpha\mu}, 
\quad \sigma_{j,k}^2\coloneqq\int |\hat \omega_{j,k}|^2 d\eta_{\alpha,\mu}(U,\omega)=\frac{j^2+k^2}{\alpha(\mu+j^2+k^2)}.
\end{equation*}
Indeed, under $\eta_{\alpha,\mu}$ the Fourier modes $\hat \omega_{j,k}$ are independent
centred Gaussian variables with the above covariances; notice that they are complex valued, but subject
to the condition $\bar{\hat \omega}_{j,k}=\hat \omega_{-j,k}$.

We have considered $\omega$ under $\eta_{\alpha,\mu}$ as a Gaussian random field indexed by $\H^0$
(its \emph{reproducing kernel Hilbert space}): it is well known that it can also be identified with
a random distribution in a larger Hilbert space into which $\H^0$ has an Hilbert-Schmidt embedding,
such as $\H^{-1-\delta}$ for any $\delta>0$.
In other terms, since all $\sigma_{j,k}^2$, $(j,k)$ varying in $\Lambda$, are of order 1,
the random Fourier series $\omega=\sum_{(j,k) \in \Lambda} \hat \omega_{j,k}e_j s_k$
converges in $L^2(\eta_{\alpha,\mu})$ in $\H^{-1-\delta}$ for any $\delta>0$, but not for $\delta\geq 0$.

\begin{lem}
	For any $\delta>0$, $(\R\times\H^{-1-\delta},\R\times\H^0,\eta_{\alpha,\mu})$ is a (complex) abstract Wiener space;
	equivalently, under $\eta_{\alpha,\mu}$, $\omega$ can be identified with a $\H^{-1-\delta}$-valued
	Gaussian random variable.
\end{lem}

We refer to \cite{DPZa14} for a complete treatment of the Gaussian analysis underlying the above discussion,
and to \cite{AlFlSi08} for its application to Enstrophy measures of 2-dimensional Euler equations.

%%%%%%%%%%%%%%%%%%%%%%%%%%%%%%%%%%%%%%%%%%%%%%%%%%%%%%%%%%%%%%%%%%%%%%%%%%%%%%%%%%%%%%%%%%%%%
%%%%%%%%%%%%%%%%%%%%%%%%%%%%%%%%%%%%%%%%%%%%%%%%%%%%%%%%%%%%%%%%%%%%%%%%%
\section{Weak Solutions for Low-Regularity Marginals}\label{sec:weaksolutions}

We now discuss how to interpret \eqref{eq:bqgreduced} in the case when, at a fixed time, 
$(U,\omega)$ is a sample of $\eta_{\alpha,\mu}$. Indeed, as we remarked above,
in that case $\omega$ can be identified at best as a distribution in $\H^{-1-\delta}$, $\delta>0$,
and thus the main concern is the nonlinear term of the evolution equation,
the affine term $L\omega$ being easily defined pathwise as a distribution of class $\H^{-2-\delta}$. 

\subsection{Fourier Expansion of the Nonlinear Term}\label{ssec:nonlinearity}
Let us fix $\delta>0$, and consider the coupling between the nonlinear term 
$\nabla^\perp \Delta^{-1}\omega \cdot \nabla\omega$ and a smooth test function $\phi \in C^\infty(D)$ . 
If $\omega\in \H^{-1-\delta}$, we can define the tensor product $\omega \otimes \omega$ as a distribution
on $D \times D$ via
\begin{equation} \label{eq:duality}
\brak{\omega \otimes \omega, \varphi \otimes \psi} \coloneqq \brak{\omega,\varphi} \brak{\omega,\psi}, \hspace{0.5cm} \varphi,\psi \in C^\infty(D),
\end{equation}
where $\varphi \otimes \psi (z,z') \coloneqq \varphi(z)\psi(z')$; it is easily observed that
the resulting distribution $\omega \otimes \omega$ is of class $\H^{-2-2\delta}(D \times D)$ (with $\H^\alpha(D\times D)$
we denote the closed subspace of $H^\alpha(D\times D)$ generated by vectors $e_j s_k \otimes e_{j'}s_{k'}$), so the expression 
\begin{equation}
\int_{D \times D} H(z,z')\omega(dz)\omega(dz')=\brak{\omega \otimes \omega, H}
\end{equation}
is well defined via duality for every $H \in \H^{2+2\delta}(D \times D)$. 
Now, given any smooth test function $\phi \in C^{\infty}(D)$, 
we look for a suitable function $H_\phi$ such that
\begin{equation} \label{eq:Hphi}
\int_D \nabla^\perp \Delta^{-1} \omega(z)\cdot\nabla\omega(z) \phi(z)dz=\int_{D \times D} H_\phi(z,z')\omega(z)\omega(z')dzdz'.
\end{equation}

We perform the computation in Fourier series. 
Let us also recall that we denote points of $D$ by $z=(x,y),z'=(x',y')$.
We thus have
\begin{gather*}
\omega(z)=\sum_{(j,k)\in\Lambda} \hat\omega_{j,k} e_j(x)s_k(y), \quad
\Delta^{-1}\omega(z)=-\sum_{(j,k)\in\Lambda} \frac{\hat\omega_{j,k}}{j^2+k^2} e_j(x)s_k(y),\\
\nabla\omega(z)=\sum_{(j,k)\in\Lambda} \binom{\imm j s_k(y)}{k c_k(y)} \hat\omega_{j,k} e_j(x),\\
\nabla^\perp\Delta^{-1}\omega(z)=\sum_{(j,k)\in\Lambda}
\binom{k c_k(y)}{-\imm j s_k(y)} \frac{\hat\omega_{j,k}}{j^2+k^2} e_j(x),\\
\nabla^\perp\Delta^{-1} \omega(z)\cdot\nabla\omega(z)=\sum_{(j,k)\in\Lambda}\sum_{(j',k')\in\Lambda}
\pa{(jk'-j'k)s_{k+k'}(y) +(j'k+jk')s_{k-k'}(y)}\\
\qquad \times \frac{\hat\omega_{j,k}\hat\omega_{j',k'}}{2\imm(j^2+k^2)} e_{j+j'}(x),
\end{gather*}
so that equation \eqref{eq:Hphi} becomes
\begin{align*}
&\int_D \nabla^\perp \Delta^{-1} \omega(z)\cdot\nabla\omega(z) \phi(z)dz\\
&=\sum_{(j,k)\in\Lambda}\sum_{(j',k')\in\Lambda} 
\pa{(jk'-j'k)\hat\phi_{-j-j',k+k'} +(j'k+jk')\hat\phi_{-j-j',k-k'}}
\frac{\hat\omega_{j,k}\hat\omega_{j',k'}}{2\imm(j^2+k^2)}\\
&=\sum_{(j,k)\in\Lambda}\sum_{(j',k')\in\Lambda} 
\pa{(jk'-j'k)\hat\phi_{-j-j',k+k'} +(j'k+jk')\hat\phi_{-j-j',k-k'}}\\
&\quad \times\pa{\frac1{j^2+k^2}-\frac1{j'^2+k'^2}}\frac{\hat\omega_{j,k}\hat\omega_{j',k'}}{4\imm}
= \sum_{\substack{(j,k)\in\Lambda\\(j',k')\in\Lambda}} \F_{-j,k}\F_{-j',k'} (H_\phi) \hat\omega_{j,k}\hat\omega_{j',k'},
\end{align*}
the second step consisting in a symmetrisation with respect to indices $(j,k)$ and $(j',k')$.
The last equality is the Fourier expansion of the right-hand side of \eqref{eq:Hphi},
and becomes our definition of $H_\phi$: 
\begin{align*}
\F_{j,k}\F_{j',k'} H_\phi \coloneqq& \pa{(j'k-jk')\hat\phi_{j+j',k+k'} - (j'k+jk')\hat\phi_{j+j',k-k'}} \\
&\times \pa{\frac1{j^2+k^2}-\frac1{j'^2+k'^2}}\frac1{4\imm},
\end{align*}
where $\F_{j,k}\F_{j',k'}$ is an abbreviation for the more rigorous notation $\F_{j,k} \otimes \F_{j',k'}$, the Fourier projector on $e_j s_k \otimes e_{j'}s_{k'}$. We also adopt the convention
\begin{equation*}
\hat\phi_{j+j',k-k'} \coloneqq -\hat\phi_{j+j',k'-k} \quad \text{whenever}\quad k-k'<0.
\end{equation*}

So far, $H_\phi$ is defined only as a formal Fourier series: the forthcoming Lemma discusses
the convergence of the latter, \emph{i.e.} the regularity of $H_\phi$.
\begin{lem} \label{prop:HphiL2}
	For every $\phi\in \H^2$, $H_\phi\in \H^0(D\times D)$.
\end{lem}

\begin{proof}
	To ease notation, we denote $l=(j,k)$ and $l'=(j',k')$. We have that $H_\phi\in \H^0(D\times D)$ if and only if 
	\begin{equation} \label{eq:000}
	\sum_{l,l' \in \Lambda} |\F_l \F_{l'} (H_\phi)|^2 < \infty.
	\end{equation}
	The Fourier coefficients of $H_\phi$ are given by two summands which we estimate separately. The first one is
	\begin{equation*}
	\F_{j+j',k+k'}(\phi) (j'k-jk') \pa{\frac1{j^2+k^2}-\frac1{j'^2+k'^2}}= {\F_{l+l'}(\phi)} \pa{-l^\perp \cdot l'}\pa{\frac1{|l|^2}-\frac1{|l'|^2}},
	\end{equation*}
	where $|l|^2=j^2+k^2$, and similarly for $l'$; taking squares and summing over $l+l'=m \in \Lambda$ gives us
	\begin{equation} \label{eq:001}
	\sum_{m \in \Lambda} |{\F_{m}(\phi)}|^2 \sum_{\substack{l \in \Lambda\\ l\neq m}} \left( l^\perp \cdot \left( m-l\right) \pa{\frac1{|l|^2}-\frac1{|m-l|^2}} \right)^2.
	\end{equation}	 
	We now resort to the following inequalities:
	\begin{align}\label{eq:007}
	l^\perp \cdot (m-l) = l^\perp \cdot m &\leq |l||m|, \\
	\label{eq:008}
	|m-l|^2 - |l|^2 = m \cdot(m-2l)\leq |m||m-2l| &\leq |m| \left( |m-l|^2 +|l|^2 \right)^{1/2},
	\end{align}
	so that the inner summation in \eqref{eq:001} can be estimated with
	\begin{align*}
	\sum_{\substack{l \in \Lambda\\ l\neq m}} \left( \frac{|l|^2 |m|^4 |m-l|^2}{|l|^4|m-l|^4} + \frac{|l|^4 |m|^4}{|l|^4|m-l|^4}\right) \nonumber
	&=|m|^4 \sum_{\substack{l \in \Lambda\\ l\neq m}}  \left(\frac{1}{|l|^2|m-l|^2} + \frac{1}{|m-l|^4} \right) \\
	&\leq 2 |m|^4 \sum_{l \in \Lambda} \frac{1}{|l|^4}. 
	\end{align*}
	Modulo a multiplicative constant, \eqref{eq:001} is therefore smaller or equal to
	\begin{equation*}
	\sum_{m \in \Lambda} |{\F_{m}(\phi)}|^2 |m|^4  ,
	\end{equation*}
	which is finite as soon as $\phi \in \H^2$. The other contribution is given by the terms of the form
	\begin{equation*}
	\F_{j+j',k-k'}(\phi) (j'k+jk')
	\pa{\frac1{j^2+k^2}-\frac1{j'^2+k'^2}},
	\end{equation*}
	which after the change of variables $(j,k,j',k') \mapsto (j,k,-j',k')$ becomes
	\begin{align*}
	\F_{l-l'}(\phi)\pa{l^\perp \cdot l'}\pa{\frac1{|l|^2}-\frac1{|l'|^2}},
	\end{align*}
	which can be estimated in a similar fashion taking the modulo square and summing over $l-l'=m \in \Lambda$. Thus \eqref{eq:000} is proved.
\end{proof}

\begin{rmk} \label{rmk:Hsin}
	Even though we will not need it in the following, for every $\delta<1$ 	the above computation actually yields $H_\phi\in \H^\delta(D\times D)$  if $\phi \in \H^{2+\delta}$.
	This in fact is the optimal Sobolev regularity, since in general $H_\phi\notin \H^\delta(D\times D)$ for $\delta \geq 1$, 
	even for more regular $\phi$. Indeed, for $\phi(x,y)=\sin(y)$ the Fourier coefficients of $H_\phi$ are given by
	\begin{equation*}
	\F_{j,k}\F_{j',k'} (H_\phi) = \frac{\mathbf{1}_{\{j+j'=0\}}  \mathbf{1}_{\{k-k'=1\}} }{4\imm}\frac{j(1-2k)}{(j^2+k^2)(j^2+(k-1)^2)},
	\end{equation*}
	therefore $H_\phi\in \H^\delta(D\times D)$ if and only if
	\begin{equation*}
	\sum_{(j,k) \in \Lambda} \left( 1+ 2j^2 +k^2 + (k-1)^2 \right)^\delta \frac{j^2(1-2k)^2}{(j^2+k^2)^2(j^2+(k-1)^2)^2} < \infty,
	\end{equation*}
	but the sum above can be estimated from below (modulo a positive multiplicative constant) by
	\begin{equation*}
	\sum_{(j,k) \in \Lambda} \frac{j^2k^2}{(j^2+k^2)^{4-\delta}},
	\end{equation*}
	the latter converging if and only if $\delta < 1$.
\end{rmk}

Unfortunately, since $H_\phi$ does not belong to $\H^{2+2\delta}(D\times D)$,
it is not possible to define the nonlinear term of \eqref{eq:bqgreduced}
\emph{pathwise}, that is fixing a realisation of $\omega$ under $\eta_{\alpha,\mu}$
and taking products of distributions. It is at this point that we make essential use of the probabilistic approach
to invariant measures.

\subsection{The Nonlinear Term as a Stochastic Integral}
Thanks to the peculiar form of the fluid-dynamic nonlinearity, which in our setting is reflected
by the coefficients of $H_\phi$, when $\omega$ is sampled from the Gaussian measure $\eta_{\alpha,\mu}$
it is possible to define the nonlinear term as a (double) stochastic integral,
that is, as an $L^2(\eta_{\alpha,\mu})$-limit of suitable approximations.

The following result finds analogues in \cite[Lemma 1.3.2]{AlCr90}, see also \cite{AlRFHK79},
and in \cite[Theorem 8]{Fl18} or the related \cite{DPFlRo17,FlLu18,Gr19,FlGrLu19}, 
all dealing with stationary solutions of 2-dimensional Euler equations.

\begin{prop}\label{prop:sim}
	Let $H\in \H^0(D\times D)$ be a symmetric function. Consider functions	
	$(H^n)_{n\in\N} \subset \H^{2+2\delta}(D \times D)$ such that,
	for every $(j,k),(j',k') \in \Lambda$,
	\begin{equation} \label{eq:002}
	\lim_{n \to \infty} \sum_{(j,k)\in \Lambda}\F_{j,k}\F_{j,k}(H^n)\sigma_{j,k}^2 =0,\quad \F_{j,k}\F_{j',k'}(H^n) =\F_{j',k'}\F_{j,k}(H^n),
	\end{equation}
	and suppose that the sequence $H^n$ approximates $H$ in the following sense:
	\begin{equation} \label{eq:003}
	\lim_{n \to \infty}  \sum_{\substack{(j,k)\in \Lambda\\(j'k') \in \Lambda}} 
	\left( \F_{j,k}\F_{j',k'}(H^n-H) \right)^2 \sigma_{j,k}^2 \sigma_{j',k'}^2 = 0.
	\end{equation}
	Under $\eta_{\alpha,\mu}$, the sequence of random variables $\brak{\omega \otimes \omega,H^n}$ 
	defined by \eqref{eq:duality}, converges in mean square. 
	Moreover, the limit does not depend on the approximating sequence $H^n$.
\end{prop}

\begin{proof}
	To ease notation we denote $l=(j,k)$ and $l'=(j',k')$. 
	For any function $H\in \H^0(D\times D)$, we compute
	\begin{align*}
	\mathbb{E} \left[ \brak{\omega \otimes \omega,H}^2 \right] &= \mathbb{E}  \left[ \sum_{\substack{l,l'\in \Lambda\\m,m' \in \Lambda}} \F_l \F_{l'} (H) \F_m \F_{m'} (H) \overline{\hat{\omega}_{l} \hat{\omega}_{l'} \hat{\omega}_{m}\hat{\omega}_{m'} } \right] \\
	&= \sum_{\substack{l,l'\in \Lambda\\m,m' \in \Lambda}} \F_l \F_{l'} (H) \F_m \F_{m'} (H)  \mathbb{E}  \left[ \overline{\hat{\omega}_{l} \hat{\omega}_{l'} \hat{\omega}_{m}\hat{\omega}_{m'} }\right].
	\end{align*}
	By Wick-Isserlis formula the expected value in the last summand is given by
	\begin{align*}
	\mathbb{E}  \left[ \overline{\hat{\omega}_{l} \hat{\omega}_{l'} \hat{\omega}_{m}\hat{\omega}_{m'} }\right] 
	&= \sigma^2_l \sigma^2_m \delta_{l,l'}\delta_{m,m'}+
	\sigma^2_l \sigma^2_{l'} \delta_{l,m}\delta_{l',m'}+\sigma^2_l \sigma^2_{l'} \delta_{l,m'}\delta_{l',m}.
	\end{align*}
	Substituting and using the relations \eqref{eq:002} one gets
	\begin{align*}
	\mathbb{E} \left[ \brak{\omega \otimes \omega,H}^2 \right] 
	= \left( \sum_{l \in \Lambda} \F_l(H) \sigma_l^2\right)^2 + 2 \sum_{l,l' \in \Lambda} \F_l \F_{l'} (H)^2  \sigma_l^2 \sigma_{l'}^2.
	\end{align*}
	If conditions \eqref{eq:002} and \eqref{eq:003} hold, applying the latter equation to differences $H^m-H^n$ we obtain that
	the sequence of random variables $\brak{\omega \otimes \omega,H^n}$ is a Cauchy sequence in $L^2(\Omega)$. 
	The independence of limit from the sequence $(H^n)$ follows from triangular inequality and \eqref{eq:003}.
\end{proof}

\begin{rmk}
	In \cite{Fl18}, conditions \eqref{eq:002} and \eqref{eq:003} are replaced by
	\begin{align*}
	&H^n \mbox{ symmetric,} \quad
	\lim_{n \to \infty} \int H^n(z,z)dz = 0, \\
	&\lim_{n \to \infty} \int \int (H^n(z,z')-H(z,z'))^2 dzdz' =0,
	\end{align*}
	where integration is performed over the 2-dimensional torus. These conditions are simpler than ours 
	since we deal with coloured noise $\eta_{\alpha,\mu}$ rather then space white noise.
\end{rmk}

Consider now a test function $\phi \in C^\infty(D)$: \autoref{prop:sim} allows us to define 
the nonlinearity \eqref{eq:Hphi} as the $L^2(\eta_{\alpha,\mu})$-limit of  
$\brak{\omega \otimes \omega,H^n_\phi}$ for any sequence $H_\phi^n$ approximating $H_\phi$ in the above sense
(for instance, progressive truncations of Fourier series). 
To emphasize the peculiarity of its definition, we adopt a special notation for this object.
\begin{definition}
	For any $H\in \H^0(D\times D)$, and $H^n$ as in \autoref{prop:sim},
	\begin{equation}\label{eq:diamond}
	\brak{\omega \diamond \omega, H} 
	\coloneqq L^2(\eta_{\alpha,\mu})-\lim_{n \to \infty} \brak{\omega \otimes \omega,H^n}.
	\end{equation}
\end{definition}
We chose a distinct symbol because if we consider a smooth $H$ and confront
the new object we define and coupling with tensor products \eqref{eq:duality},
a straightforward computation reveals that
\begin{equation*}
\brak{\omega \diamond \omega, H}=\brak{\omega \otimes \omega, H} -  \sum_{(j,k)\in \Lambda}\F_{j,k}\F_{j,k}(H)\sigma_{j,k}^2.
\end{equation*}
Indeed, let $H^n$ be the following approximation of $H$:
\begin{align*}
\F_{j,k}\F_{j',k'}(H^n) \coloneqq \F_{j,k}\F_{j',k'}(H) - \frac{S}{n \sigma_{0,k}^2} \mathbf{1}_{\{j=j'=0,k=k'=1,\dots,n\}},
\end{align*}
where $S \coloneqq \sum_{(j,k)\in \Lambda}\F_{j,k}\F_{j,k}(H)\sigma_{j,k}^2 < \infty$. Hence 
\begin{equation*}
\brak{\omega \diamond \omega, H}=\brak{\omega \otimes \omega, H} - \lim_{n\to \infty} \sum_{k=1}^n \hat{\omega}_{0,k}^2 \frac{S}{n \sigma_{0,k}^2} = \brak{\omega \otimes \omega, H} - S.
\end{equation*}
as an equality between random variables in $L^2(\eta_{\alpha,\mu})$.
Notice that the last summand in the latter expression diverges for a generic $H\in \H^0(D\times D)$,
according to the fact that the coupling with tensor product $\omega \otimes \omega$ can not be defined in that case.

\begin{rmk}
	The present paragraph takes its name because the coupling $\brak{\omega \diamond \omega, H}$
	we defined in fact corresponds to the double It\=o-Wiener integral of $H$ with respect to
	the Gaussian measure $\eta_{\alpha,\mu}$.
\end{rmk}

We now extend \autoref{prop:sim} to manage stochastic processes, rather than just random variables.
\begin{prop} \label{prop:Franco}
	On a probability space $(\Omega,\F,\PP)$ consider a stochastic process $(\omega_t)_{t\in[0,T]}$
	with trajectories in $C([0,T],\H^{-1-\delta})$ such that the law of $\omega_t$ is $\eta_{\alpha,\mu}(d\omega)$ 
	for every $t \in [0,T]$. 
	Let $(H^n_\phi)_{n\in\N} \subseteq \H^{2+2\delta}(D \times D)$ be an approximation of $H_\phi$ in the sense of \autoref{prop:sim}. 
	Then the sequence of processes $t \mapsto \brak{\omega_t \otimes \omega_t,H^n_\phi}$ converges in $L^2([0,T],L^2(\PP))$. 
	Moreover, the limit does not depend on the approximating functions $H^n_\phi$.
\end{prop}

The proof is a direct consequence of \autoref{prop:sim} and stationarity of the process $\omega$.
We are now ready to give the definition of solution we mentioned in \autoref{thm:maintheorem}.

\begin{definition} \label{def:bqgvorticity}
	A stochastic process $(U_t,\omega_t)_{t\in[0,T]}$ defined on a probability space $(\Omega,\F,\PP)$, 
	with trajectories in $C([0,T];\R \times \H^{-1-\delta})$, solves the reduced form \eqref{eq:bqgreduced}
	of \eqref{eq:bqg} in the \emph{weak vorticity formulation} if,
	for every test function $\phi \in C^\infty(D)$, $\PP$-almost surely, for every $t\in [0,T]$,
	\begin{gather}\label{eq:bqgweak}
	\brak{\omega_t, \phi} = \brak{\omega_0, \phi} + \int_0^t \brak{\omega_s \diamond \omega_s, H_\phi} \,ds + \int_0^t \brak{L \omega_s, \phi} \,ds,\\ \label{eq:bqgweaku}
	U_t-U_0=\int_0^t \dashint_D h \left( \de_x \Delta^{-1}\omega_s+\frac{\de_x}{\mu-\Delta}h\right)dzds,
	\end{gather}
	where the process $s \mapsto \brak{\omega_s \diamond \omega_s, H_\phi} $ is defined by \autoref{prop:Franco}.
\end{definition}

In the remainder of the paper we will focus on equations for centred variables $(U,\omega)$, 
and thus prove the following corresponding version of \autoref{thm:maintheorem}, from which
the latter is straightforwardly recovered.

\begin{thm}\label{thm:maintheoremcentred}
	Let $\beta\neq 0$ and $h$ as above. For any $\alpha,\mu>0$ there exists a stationary stochastic process 
	$(U_t,\omega_t)_{t\in [0,T]}$ with trajectories in $C([0,T],\R \times \H^{-1-\delta})$
	and fixed-time marginals $\eta_{\alpha,\mu}$,
	whose trajectories solve \eqref{eq:bqgreduced} in the weak vorticity formulation of \autoref{def:bqgvorticity}.
\end{thm}

%%%%%%%%%%%%%%%%%%%%%%%%%%%%%%%%%%%%%%%%%%%%%%%%%%%%%%%%%%%%%%%%%%%%%%%%%%%%%%%%%%%%%%%%%%%%%
\section{A Galerkin Approximation Scheme}\label{sec:galerkin}

Let us define the finite-dimensional projection of $L^2(D)$ 
onto the finite set of modes $\Lambda_N=\set{(j,k)\in\Lambda:j^2+k^2 \leq N}$,
\begin{equation}
\Pi_N:L^2(D)\ni f\mapsto \Pi_N f \coloneqq \sum_{(j,k)\in\Lambda_N} \F_{j,k}(f) e_j s_k \in \H_N,
\end{equation}
where we can identify the finite dimensional codomain with
\begin{align*}
\H_N&=\set{\sum_{(j,k)\in\Lambda_N} \xi_{j,k} e_j s_k\,:\, \xi_{j,k}=\overline{\xi_{-j,k}}}
\simeq \set{\xi\in \C^{\Lambda_N}: \xi_{j,k}=\overline{\xi_{-j,k}}} 
\simeq \C^{\tilde\Lambda_N},\\
\tilde\Lambda_N&=\set{(j,k)\in\Lambda:j\geq 0,j^2+k^2 \leq N}.
\end{align*}

%%%%%%%%%%%%%%%%%%%%%%%%%%%%%%%%%%%%%%%%%%%%%%%%%%%%%%%%%%%%%%%%%%%%%%%%%%%%%%%%%%%%%%%%%%%%
\subsection{Truncated Barotropic Quasi-Geostrophic Equations}

Let $h^N=\Pi_N h$: we consider the following truncated version of \eqref{eq:bqgreduced},
\begin{equation} \label{eq:truncbqg}
\begin{cases}
\de_t\omega^N+ \Pi_N \left( \nabla^\perp \Delta^{-1}\omega^N\cdot \nabla\omega^N \right) + L_N\omega^N=0\\
\frac{d U^N}{dt}=\dashint_D h^N \de_x \Delta^{-1}\omega^N,
\end{cases}
\end{equation}
with $L_N\omega^N$ collecting affine terms in $\omega^N$:
\begin{align*}
L_N\omega^N &= \pa{U^N-\frac\beta\mu}\de_x\omega^N
+U^N \frac{\mu\de_x}{\mu-\Delta}h^N
+\Pi_N \left( \frac{\nabla^\perp}{\mu-\Delta}h^N\cdot\nabla\omega^N\right) \\
&+\Pi_N \left( \nabla^\perp\Delta^{-1}\omega^N \cdot \frac{\mu\nabla}{\mu-\Delta}h^N	\right)
+\beta \de_x\Delta^{-1}\omega^N.
\end{align*}
For the sake of simplicity, equations \eqref{eq:truncbqg} can be rewritten in the compact form
\begin{equation}
\de_t(U^N, \omega^N )= B^N(U^N, \omega^N),
\end{equation}
where $\omega^N$ is the vector with components $(\hat{\omega}^N_{j,k})_{(j,k)\in \tilde\Lambda_N}$,
and $B^N:\R\times \H_N\rightarrow \R\times \H_N$. Let us stress the fact that we can
reduce ourselves to consider Fourier modes in $\tilde\Lambda_N$ thanks to $\hat{\omega}^N_{j,k}=\obar{\hat{\omega}^N_{-j,k}}$.

Galerkin approximants \eqref{eq:truncbqg} are globally well-posed, and truncation is such that they preserve
the following projection of $\eta_{\alpha,\mu}$,
\begin{equation*}
\eta_{\alpha,\mu}^N=(\id_\R,\Pi_N)_\#\eta_{\alpha,\mu}.
\end{equation*}
In other words, under $\eta_{\alpha,\mu}^N$, $U^N$ has the same Gaussian distribution
of $U$ under $\eta_{\alpha,\mu}$, while $\omega^N$ is the projection of $\omega$ under $\eta_{\alpha,\mu}$.
More explicitly, we can define $\eta_{\alpha,\mu}^N$ by density with respect to the product Lebesgue measure
on $\H_N\simeq \tilde \Lambda_N$,
\begin{align*}
d\eta_{\alpha,\mu}^N(U^N,\omega^N)&=\frac1{Z^N_{\alpha,\mu}} 
e^{-\frac{\alpha\mu}{2}(U^N)^2}dU^N\\
&\qquad \times \prod_{(j,k)\in\tilde\Lambda_N} \exp\pa{-\frac\alpha2 |\hat{\omega}^N_{j,k}|^2
	\pa{1+\frac{\mu}{j^2+k^2}}} d\hat{\omega}^N_{j,k}.
\end{align*}

\begin{prop}\label{prop:galerkinwellposed}
	For $\eta_{\alpha,\mu}^N$-almost every initial datum $(U^N_0,\omega^N_0)$,
	there exists a unique solution $(U^N_t,\omega^N_t)\in C^\infty ([0,\infty),\R\times \H_N)$
	to the ordinary differential equation \eqref{eq:truncbqg}. 
	Moreover, the global flow preserves $\eta_{\alpha,\mu}^N$.
\end{prop}

\begin{proof}
	The components of vector field $B^N$ are polynomials of $U^N,\hat{\omega}^N_{j,k}$, $(j,k)\in\tilde\Lambda_N$,
	and thus $B^N$ and its derivatives have finite moments of all orders under $\eta_{\alpha,\mu}^N$.	
	The thesis then follows from non-explosion results in \cite[Section 3]{Cr83}, as soon as we check
	that $B^N$ has null divergence with respect to $\eta_{\alpha,\mu}^N$, \emph{i.e.}
	\begin{align*}
	0=\div_{\eta_{\alpha,\mu}^N} B^N&=\de_{U_N}B^N_{U_N} +\sum_{(j,k)\in\tilde\Lambda_N} \de_{j,k} B^N_{j,k}\\
	&\qquad -\alpha\mu U_N B^N_{U_N}	
	-\alpha \sum_{(j,k)\in\tilde\Lambda_N}\pa{1+\frac{\mu}{j^2+k^2}} \hat{\omega}^N_{j,k} B^N_{j,k},
	\end{align*}
	subscripts denoting components (and derivatives) relative to $U_N$ or $\hat\omega_{j,k}$.
	In fact, \cite{Cr83} treats the case of a standard Gaussian measure on $\R^n$, but 
	their results are easily extended to our case.
	Showing that $B^N$ is divergence-free with respect to $\eta_{\alpha,\mu}^N$ can be done by
	direct computation: the full computation in the case of completely periodic geometry
	can be found in \cite[Section 6.2]{MaTiVE01}, to which we refer, the differences with
	our case being minimal.
\end{proof}

%%%%%%%%%%%%%%%%%%%%%%%%%%%%%%%%%%%%%%%%%%%%%%%%%%%%%%%%%%%%%%%%%%%%%%%%%%%
\subsection{The Truncated Nonlinear Term} \label{ssec:nonlinearitygalerkin}
In the finite-dimensional Galerkin truncation \eqref{eq:truncbqg} we can repeat the arguments
of \autoref{ssec:nonlinearity} to cast couplings of the nonlinear term into a double integral
formulation. For any $\phi \in C^\infty(D)$, expanding in Fourier series the equality
\begin{align*}
\brak{\Pi_N \pa{\nabla^\perp \Delta^{-1} \omega^N \cdot \nabla \omega^N},\phi}_{\H^0} 
&= \brak{\nabla^\perp \Delta^{-1} \omega^N \cdot \nabla \omega^N,\Pi_N\phi}_{\H^0}\\
&= \brak{\omega^N \otimes \omega^N,H^N_\phi}_{\H^0(D\times D)},
\end{align*}
%\begin{align*}	
%	&\sum_{\substack{(j,k)\in\Lambda_N\\(j',k')\in\Lambda_N\\(j+j',k+k') \in \Lambda_N}}
%	{(jk'-j'k)}\hat\phi_{-j-j',k+k'} \frac{\hat\omega_{j,k}\hat\omega_{j',k'}}{2\imm(j^2+k^2)}\\
%	+&\sum_{\substack{(j,k)\in\Lambda_N\\(j',k')\in\Lambda_N\\(j+j',k-k') \in \Lambda_N}}(j'k+jk')\hat\phi_{-j-j',k-k'} 
%	\frac{\hat\omega_{j,k}\hat\omega_{j',k'}}{2\imm(j^2+k^2)}\\
%	=&\sum_{\substack{(j,k)\in\Lambda_N\\(j',k')\in\Lambda_N\\(j+j',k+k') \in \Lambda_N}}
%	{(jk'-j'k)}\hat\phi_{-j-j',k+k'}\pa{\frac1{j^2+k^2}-\frac1{j'^2+k'^2}}\frac{\hat\omega_{j,k}\hat\omega_{j',k'}}{4\imm}\\
%	+&\sum_{\substack{(j,k)\in\Lambda_N\\(j',k')\in\Lambda_N\\(j+j',k-k') \in \Lambda_N}}(j'k+jk')\hat\phi_{-j-j',k-k'} 
%	\pa{\frac1{j^2+k^2}-\frac1{j'^2+k'^2}}\frac{\hat\omega_{j,k}\hat\omega_{j',k'}}{4\imm},
%\end{align*}
%where $\hat{\phi}_{j,k}= \F_{j,k}(\phi)$ is the Fourier coefficients of $\phi$ with respect to the basis vector $e_j s_k$. Notice that the components of $\phi$ on the basis vector $e_j c_k$ do not give any contribution to \eqref{eq:nonlin}.  
we deduce a Fourier expansion of $H^N_\phi$,
%with respect to the vectors $\{e_j s_k \otimes e_{j'}s_{k'}\}_{(j,k),(j',k') \in \Lambda_N}$:
\begin{align*}
\F_{{j},k}\F_{{j'},k'} H^N_\phi&={(j'k-jk')}\hat\phi_{{j+j'},k+k'} \frac{\mathbf{1}_{\{(j+j',k+k') \in \Lambda_N\}}}{4\imm}\pa{\frac1{j^2+k^2}-\frac1{j'^2+k'^2}} \\
&- (j'k+jk')\hat\phi_{{j+j'},k-k'}\frac{\mathbf{1}_{\{(j+j',k-k') 
		\in \Lambda_N\}}}{4\imm}\pa{\frac1{j^2+k^2}-\frac1{j'^2+k'^2}},	
\end{align*}
the computation being completely analogous to the one in \autoref{ssec:nonlinearity}.

%%%%%%%%%%%%%%%%%%%%%%%%%%%%%%%%%%%%%%%%%%%%%%%%%%%%%%%%%%%%%%%%%%%%%%
\subsection{Compactness Results}\label{ssec:simon}
The first step towards taking the limit of Galerkin approximants as $N\rightarrow\infty$
is to provide estimates from which we can deduce relative compactness of approximations.

We begin by reviewing a deterministic compactness criterion due to Simon,
which allows us to control separately time and space regularity,
in the spirit of Aubin-Lions compactness Lemma.
We refer to \cite{Si87} for the result and the required generalities on Banach-valued Sobolev spaces.

\begin{prop}[Simon]
	Assume that
	\begin{itemize}
		\item $X\hookrightarrow B\hookrightarrow Y$ are Banach spaces such that the embedding $X\hookrightarrow Y$
		is compact and there exists $0<\theta<1$ such that for all $v\in X\cap Y$
		\begin{equation*}
		\norm{v}_B\leq M\norm{v}_X^{1-\theta} \norm{v}_Y^\theta;
		\end{equation*}
		\item $s_0,s_1\in\R$ are such that $s_\theta=(1-\theta)s_0+\theta s_1>0$.
	\end{itemize}
	If $\F \subset W$ is a bounded family in
	\begin{equation*}
	W=W^{s_0,r_0}([0,T],X)\cap W^{s_1,r_1}([0,T],Y)
	\end{equation*}
	with $r_0,r_1\in[0,\infty]$, and moreover
	\begin{equation*}
	s^*=s_\theta-\frac{1-\theta}{r_0}-\frac{\theta}{r_1}>0,
	\end{equation*}
	then if $\F$ is relatively compact in $C([0,T],B)$.
\end{prop}

Let us specialise this result to our framework. Take
\begin{equation*}
X=\R \times \H^{-1-\delta/2}, \quad B=\R \times \H^{-1-\delta}, \quad Y=\R \times \H^{-3-\delta},
\end{equation*}
with $\delta>0$: by Gagliardo-Niremberg estimates the interpolation inequality is satisfied with $\theta=\delta/2$.
Let us take moreover $s_0=0$, $s_1=1$, $r_1=2$ and $r_0=q\geq1$; if we can take $q$ large such that
\begin{equation*}
s^*=\frac\delta4-\frac{2-\delta}{2q}>0,
\end{equation*}
then the hypothesis are satisfied and obtain:
\begin{cor}\label{cor:simon}
	Let $\delta>0$. If a family of functions
	\begin{equation*}
	\set{v_n}\subset L^q([0,T],\R \times \H^{-1-\delta/2})\cap W^{1,2}([0,T],\R \times \H^{-3-\delta})
	\end{equation*}
	is bounded for any $q\geq 1$, then it is relatively compact in 	$C([0,T],\R \times \H^{-1-\delta})$.
	
	As a consequence, if a sequence of stochastic processes $u^n:[0,T]\rightarrow \R \times \H^{-1-\delta}$, $n\in\N$, 
	defined on a probability space $(\Omega,\F,\PP)$ is such that, for any $q\geq1$,
	there exists a constant $C_{T,\delta,q}$ for which
	\begin{equation}\label{eq:momentscondition}
	\sup_n \expt{\norm{u^n(t)}^p_{ L^q([0,T],\R \times \H^{-1-\delta/2})}+\norm{u^n}_{W^{1,2}([0,T],\H^{-3-\delta})}}
	\leq C_{T,\delta,q},
	\end{equation} 
	then the laws of $u^n$ on $C([0,T],\R \times \H^{-1-\delta})$ are tight.
\end{cor}
For the sake of completeness we remark that the second, probabilistic part of the latter
statement follows from the deterministic one and a simple application of Chebyshev inequality. 

We want to apply \autoref{cor:simon} to the sequence of finite dimensional Galerkin approximations
we built in \autoref{prop:galerkinwellposed}.
To obtain the uniform bound \eqref{eq:momentscondition}, let us begin with the ``space regularity'' part: 
by stationarity of the process $(U^N,\omega^N)$ we can swap expectations and time integrals, so that
\begin{align*}
&\expt{\norm{U^N}^p_{L^q([0,T])} +\norm{\omega^N}^p_{L^q([0,T],\H^{-1-\delta/2})}}
\leq T \expt{\abs{U^N}^p +\norm{\omega^N}^p_{\H^{-1-\delta/2}}}\\
&\qquad \leq T \int \pa{\abs{U}^p+\norm{\omega}^p_{\H^{-1-\delta/2}}}d\eta_{\alpha,\mu}(dU,d\omega)
\leq C_{T,p,\alpha,\mu}.
\end{align*}
As for bounds on time regularity: starting with $U^N$, by the evolution equation
\begin{align*}
&\|U^N\|^2_{W^{1,2}([0,T])} = 
\|U^N\|^2_{L^2([0,T])}+ \left\|\frac{dU^N}{dt}\right\|^2_{L^2([0,T])}   \\
&\leq \int_0^T \left( |U^N_t|^2 + \left| \dashint_D h^N  \partial_x \Delta
^{-1}\omega^N_t  \right|^2 \right)\,dt,
\end{align*}
from which we deduce, using that $\omega^N$ has marginals $\eta^N_{\alpha,\mu}(d\omega^N)$ for every fixed time $t$,
and that $h\in C^\infty(D)$,
\begin{align*}
\mathbb{E} \left[ \|U^N\|^2_{W^{1,2}([0,T])} \right] 
&\leq C_{T,h} \left( 1+ \mathbb{E} \left[ \|\omega^N\|^2_{\H^{-1-\delta}} \right] \right) \\
&\leq  C_{T,h} \left( 1+ \mathbb{E} \left[ \|\omega\|^2_{\H^{-1-\delta}} \right] \right)
\leq C_{T,\alpha,\mu,h}. 
\end{align*}
Let us now focus on time regularity of $\omega^N$: we have
\begin{align*}
&\|\omega^N\|^2_{W^{1,2}([0,T],\H^{-3-\delta})} = \|\omega^N\|^2_{L^2([0,T],\H^{-3-\delta})}+ \|\de_t\omega^N\|^2_{L^2([0,T],\H^{-3-\delta})} \\
&\leq 2 \int_0^T \left( \|\omega^N_t\|^2_{\H^{-3-\delta}} + \| \Pi_N \left( \nabla^\perp \Delta^{-1} \omega^N_t \cdot \nabla\omega^N_t \right) \|^2_{\H^{-3-\delta}} + \|L_N \omega^N_t\|^2_{\H^{-3-\delta}}\right)\,dt.
\end{align*}
The affine term is controlled at any fixed time $t$ by
\begin{align*}
\mathbb{E} \left[ \|L_N \omega^N_t\|^2_{\H^{-3-\delta}} \right]& 
\leq  C_{\alpha,\mu,h} \left( 1+ \mathbb{E} \left[ \|\omega^N\|^2_{\H^{-1-\delta}} \right] \right) \\
&\leq  C_{\alpha,\mu,h} \left( 1+ \mathbb{E} \left[ \|\omega\|^2_{\H^{-1-\delta}} \right] \right)\leq C_{\alpha,\mu,h}.
\end{align*}
The quadratic term is the one forcing us to consider a large Hilbert space such as $\H^{-3-\delta}$.
As above, we denote $m = (j,k) \in \Lambda_N$. We set $\phi_m = e_j s_k$ and consider
\begin{equation*}
\mathbb{E} \left[ \brak{\omega^N \otimes \omega^N, H^N_{\phi_m}}^2\right] = 
\mathbb{E} \left[ \left( \sum_{l,l' \in \Lambda_N} \F_{l} \F_{l'} (H_{\phi_m}) 
\overline{\hat\omega^N_l \hat\omega^N_{l'}} \right)^2 \right], 
\end{equation*}
where, by the expansion we derived in \autoref{ssec:nonlinearitygalerkin},
\begin{equation*}
\F_{l} \F_{l'} (H^N_{\phi_m}) = -l^\perp \cdot l' \frac{\mathbf{1}_{\{l+l'=m\}}}{4\imm} \left( \frac{1}{|l|^2} - \frac{1}{|l'|^2} \right) + l^\perp \cdot l' \frac{\mathbf{1}_{\{l-l'=m\}}}{4\imm} \left( \frac{1}{|l|^2} - \frac{1}{|l'|^2} \right).
\end{equation*}
We can consider only the first contribution of the latter sum, 
since, up to a constant, we can bound the contribution of the sum with the contributions of the sole first term, 
similarly to what we did in the proof of \autoref{prop:HphiL2}. We obtain:
\begin{align}\label{eq:005}
&\expt{\brak{\omega^N \otimes \omega^N, H^N_{\phi_m}}^2}
\leq C \sum_{\substack{l,h \in \Lambda_N \\ l,h \neq m}} 
l^\perp \cdot (m-l) \left(\frac{1}{|l|^2} - \frac{1}{|m-l|^2} \right)\\ \nonumber
&\qquad \times h^\perp \cdot (m-h) \left(\frac{1}{|h|^2} - \frac{1}{|m-h|^2} \right)
\mathbb{E} \left[ \overline{\hat\omega^N_l \hat\omega^N_{m-l}  \hat\omega^N_{h}  \hat\omega^N_{m-h}} \right].
\end{align}
By Wick-Isserlis Formula the expected value on the right-hand side is given by
\begin{align}\nonumber
&\mathbb{E} \left[ \overline{\hat\omega^N_l \hat\omega^N_{m-l}  \hat\omega^N_{h}  \hat\omega^N_{m-h}} \right]\\ 
\nonumber
&\qquad= \sigma^2_l \sigma^2_h \delta_{l,m-l}\delta_{h,m-h}+
\sigma^2_l \sigma^2_{m-l} \delta_{l,h}\delta_{m-l,m-h}+\sigma^2_l \sigma^2_h \delta_{l,m-h}\delta_{m-l,h} \\
\label{eq:006}
&\qquad= \sigma^2_l \sigma^2_h \delta_{l,m-l}\delta_{h,m-h}+
\sigma^2_l \sigma^2_{m-l} \delta_{l,h}+\sigma^2_l \sigma^2_h \delta_{l,m-h}.
\end{align}
Notice that if $l = m-l$ we have $l^\perp (m-l)=0$, 
hence the first summand in \eqref{eq:006} does not play any role in the computation of \eqref{eq:005}.
Moreover, it is easy to check that the second and third terms give the same contribution, since $l^\perp \cdot h = - h^\perp \cdot l$. Therefore, applying inequalities \eqref{eq:007},\eqref{eq:008},
\begin{align*}
\mathbb{E} \left[ \brak{\omega^N \otimes \omega^N, H^N_{\phi_m}}^2\right] 
&\leq C \sum_{\substack{l \in \Lambda_N\\l \neq m}} \sigma^2_l \sigma^2_{m-l} \left(l^\perp \cdot (m-l) \pa{\frac{1}{|l|^2} - \frac{1}{|m-l|^2}} \right)^2\\
&\leq C\sum_{\substack{l \in \Lambda_N\\ l\neq m}}  \sigma^2_l \sigma^2_{m-l} \left( \frac{|l|^2 |m|^4 |m-l|^2}{|l|^4|m-l|^4} + \frac{|l|^4 |m|^4}{|l|^4|m-l|^4}\right)\\
&=C|m|^4 \sum_{\substack{l \in \Lambda_N\\ l\neq m}}  \sigma^2_l \sigma^2_{m-l} \left(\frac{1}{|l|^2|m-l|^2} + \frac{1}{|m-l|^4} \right) \\
&\leq C |m|^4 \sum_{l \in \Lambda_N} \frac{ \sigma^2_l \sigma^2_{m-l}}{|l|^4}. 
\end{align*}
Recall now the expression for $\sigma^2_l$:
\begin{align*}
\sigma^2_l = \frac{|l|^2}{\alpha(\mu+|l|^2)},
\end{align*}
which is smaller than $\alpha^{-1}$ for every $l$. Therefore
$ \sum_{l \in \Lambda_N} \frac{ \sigma^2_l \sigma^2_{m-l}}{|l|^4}$ is bounded form above uniformly in $m\in \Lambda_N, N \in \mathbb{N}$. Hence
\begin{align*}
\mathbb{E} \left[ \| \Pi_N \left( \nabla^\perp \Delta^{-1} \omega^N \cdot \nabla \omega^N \right)\|^2_{\H^{-3-\delta}}\right] &\leq C \sum_{m \in \Lambda_N} \frac1{(1+|m|^2)^{3+\delta}} \mathbb{E} \left[ \brak{\omega^N \otimes \omega^N, H^N_{\phi_m}}^2\right] \\
&\leq   c_{\delta} \sum_{m \in \Lambda_N} \frac{|m|^4}{(1+|m|^2)^{3+\delta}} \leq C_{\delta},
\end{align*}
where $C_{\delta}$ is a finite constant which does not depend on $N$. All in all, we arrive to
\begin{align*}
\mathbb{E} \left[ \|\omega^N\|^2_{W^{1,2}([0,T],\H^{-3-\delta})} \right] \leq C_{T,\alpha,\mu,h} \left( 1 + \mathbb{E} \left[ \|\omega\|^2_{\H^{-1-\delta}}\right] \right).
\end{align*}
The estimates made so far, combined with \autoref{cor:simon}, lead us finally to:
\begin{lem}
	The laws $\Theta^N_{\alpha,\mu}$ of the sequence of processes $u^N=(U^N_t,\omega^N_t)_{t\in T}$ defined by \autoref{prop:galerkinwellposed} are tight on $C([0,T],\R \times \H^{-1-\delta})$.
\end{lem}

\subsection{The Continuous Limit}

By Prokhorov theorem there exists a subsequence of $\Theta_{\alpha,\mu}^N$ 
--with a slight abuse of notation we will denote it with the same symbol-- 
weakly converging to a probability measure $\Theta_{\alpha,\mu}$ on $C([0,T],\R \times \H^{-1-\delta})$. 
By Skorokhod theorem, there exists a new probability space $(\tilde{\Omega},\tilde{\mathcal{F}},\tilde{\PP})$ 
and random variables $\tilde{u}^N$, $\tilde{u}$ with values in $C([0,T],\R \times \H^{-1-\delta})$ such that:
\begin{itemize}
	\item the law of $\tilde{u}^N$ (resp. $\tilde u$) is $\Theta_{\alpha,\mu}^N$ (resp. $\Theta_{\alpha,\mu}$);
	\item  $\tilde{u}^N$ converges to  $\tilde u$ $\tilde{\PP}$-almost surely. 
\end{itemize}
In order to lighten notation, we will drop tilde superscripts in the following.

The aim of this final paragraph is to prove that the stochastic process $u$ is a weak solution of \eqref{eq:bqg} 
in the sense of  \autoref{def:bqgvorticity}, thus concluding the proof of \autoref{thm:maintheoremcentred}. 
First of all, we make the following fundamental observation.

\begin{lem}
	The Galerkin approximations $u^N = (U^N,\omega^N)$ solve \eqref{eq:truncbqg} in the sense of \autoref{def:bqgvorticity}. More precisely, given any test function $\phi \in C^\infty(D)$,
	\begin{equation} \label{eq:004}
	\brak{\omega^N_t,\phi} = \brak{\omega^N_0,\phi} + \int_0^t \brak{\omega_s^N \otimes \omega_s^N, H_{\phi}} \,ds + \int_0^t \brak{L_N\omega_s^N , {\phi}} \,ds,
	\end{equation}
\end{lem}

\begin{proof}
	This follows from the discussion made in \autoref{ssec:nonlinearity}.
\end{proof}

\begin{proof}[Proof of \autoref{thm:maintheoremcentred}]
	All but the bilinear term in \eqref{eq:004} converge almost surely 
	because of the convergence of $\omega^N \to \omega$ in $C([0,T],\H^{-1-\delta})$
	and continuity of duality coupling with $\phi$. 
	The almost sure convergence of $U^N$ to $U$ solving \eqref{eq:bqgweaku} follows similarly. 
	Let us thus focus on convergence of the nonlinearity. 
	For any given $\phi \in C^\infty(D)$ and $M \in \N$ it holds
	\begin{align*}
	\int_0^t \brak{\omega_s^N \otimes \omega_s^N, H_{\phi}} \,ds &= 
	\int_0^t \brak{\omega_s^N \otimes \omega_s^N,H_\phi - H^M_\phi} \,ds \\
	&\quad+  \int_0^t \brak{\omega_s^N \otimes \omega_s^N - \omega_s \otimes \omega_s,H^M_\phi} \,ds \\
	&\quad +  \int_0^t \brak{\omega_s \otimes \omega_s,H^M_\phi} \,ds.
	\end{align*}
	For the first term on the right-hand side we have the following $L^1$ estimate: 
	\begin{align*}
	\mathbb{E}& \left[ | \brak{\omega^N \otimes \omega^N ,H_\phi - H_\phi^M}|\right] \leq \mathbb{E} \left[  \sum_{m \in \Lambda} |\hat{\phi}_m |  \left|\brak{\omega^N \otimes \omega^N ,H_{\phi_m} - H_{\phi_m}^M}\right|\right] \\
	&\leq \left( \sum_{m \in \Lambda}|\hat{\phi}_m|^2 (1+|m|^2)^\beta \right)^{1/2} 
	\left(\sum_{m \in \Lambda} \frac{\mathbb{E} \left[\left| \brak{\omega^N \otimes \omega^N ,H_{\phi_m} - H_{\phi_m}^M} \right|\right]^2}{(1+|m|^2)^\beta} \right)^{1/2} \\
	&\leq \left( \sum_{m \in \Lambda}|\hat{\phi}_m|^2 (1+|m|^2)^\beta \right)^{1/2} 
	\left(\sum_{\substack{m \in \Lambda\\m \notin \Lambda_M}} \frac{\mathbb{E} \left[ \brak{\omega^N \otimes \omega^N ,H_{\phi_m}}^2\right]}{(1+|m|^2)^\beta} \right)^{1/2} \\
	&\leq C \|\phi\|_{H^\beta} \left(\sum_{\substack{m \in \Lambda\\m \notin \Lambda_M}} \frac{|m|^4}{(1+|m|^2)^\beta} \right)^{1/2} \to 0 \mbox{ as $M \to \infty$ for }\beta > 3.
	\end{align*}
	For the last term, \autoref{prop:Franco} implies the convergence in $L^2([0,T],L^2(\Omega))$
	\begin{equation*}
	\int_0^t \brak{\omega_s \otimes \omega_s, H^M_{\phi}} \,ds \to \int_0^t \brak{\omega_s \diamond \omega_s, H_{\phi}} \,ds
	\end{equation*}
	as long as we check that $H^M_\phi$ is an approximation of $H_\phi$ in the sense of \autoref{prop:sim}. 
	But this last property is easily implied by the definition of $H^M_{\phi}$ and \autoref{prop:HphiL2}. 
	The second term in the right-hand side goes to zero as $N \to \infty$ for every fixed $M$,
	since $\omega^N \otimes \omega^N$ converges almost surely to $\omega \otimes \omega$ in $C([0,T],\H^{-2-2\delta}(D\times D))$,
	and $H^M_\phi$ belongs to $C^\infty(D\times D)$. Thus, up to subsequences, we have the almost sure convergence
	\begin{equation*}
	\int_0^t \brak{\omega_s^N \otimes \omega_s^N, H_{\phi}} \,ds \to \int_0^t\brak{\omega_s \diamond \omega_s, H_{\phi}} \,ds.
	\end{equation*} 
	Therefore, taking the almost sure limit in \eqref{eq:004} we get
	\begin{equation*}
	\brak{\omega_t,\phi} = \brak{\omega_0,\phi} + \int_0^t\brak{\omega_s \diamond \omega_s, H_{\phi}} \,ds. 
	+ \int_0^t \brak{L\omega_s , {\phi}} \,ds.\qedhere
	\end{equation*}
\end{proof}

\bibliographystyle{plain}

\begin{thebibliography}{10}
	
	\bibitem{AlRFHK79}
	S.~Albeverio, M.~Ribeiro~de Faria, and R.~H{\"o}egh-Krohn.
	\newblock Stationary measures for the periodic {E}uler flow in two dimensions.
	\newblock {\em J. Statist. Phys.}, 20(6):585--595, 1979.
	
	\bibitem{AlCr90}
	Sergio Albeverio and Ana~Bela Cruzeiro.
	\newblock Global flows with invariant ({G}ibbs) measures for {E}uler and
	{N}avier-{S}tokes two-dimensional fluids.
	\newblock {\em Comm. Math. Phys.}, 129(3):431--444, 1990.
	
	\bibitem{AlFlSi08}
	Sergio Albeverio, Franco Flandoli, and Yakov~G. Sinai.
	\newblock {\em S{PDE} in hydrodynamic: recent progress and prospects}, volume
	1942 of {\em Lecture Notes in Mathematics}.
	\newblock Springer-Verlag, Berlin; Fondazione C.I.M.E., Florence, 2008.
	\newblock Lectures given at the C.I.M.E. Summer School held in Cetraro, August
	29--September 3, 2005, Edited by Giuseppe Da Prato and Michael R\"{o}ckner.
	
	\bibitem{AzBe15}
	Jonas Azzam and Jacob Bedrossian.
	\newblock Bounded mean oscillation and the uniqueness of active scalar
	equations.
	\newblock {\em Trans. Amer. Math. Soc.}, 367(5):3095--3118, 2015.
	
	\bibitem{Ch19a}
	Qingshan Chen.
	\newblock The barotropic quasi-geostrophic equation under a free surface.
	\newblock {\em SIAM J. Math. Anal.}, 51(3):1836--1867, 2019.
	
	\bibitem{Ch19b}
	Qingshan Chen.
	\newblock On the well-posedness of the inviscid multi-layer quasi-geostrophic
	equations.
	\newblock {\em Discrete Contin. Dyn. Syst.}, 39(6):3215--3237, 2019.
	
	\bibitem{ChZMGh03}
	Zhi-Min Chen, Michael Ghil, Eric Simonnet, and Shouhong Wang.
	\newblock Hopf bifurcation in quasi-geostrophic channel flow.
	\newblock {\em SIAM J. Appl. Math.}, 64(1):343--368, 2003.
	
	\bibitem{Cr83}
	Ana~Bela Cruzeiro.
	\newblock \'{E}quations diff\'{e}rentielles ordinaires: non explosion et
	mesures quasi-invariantes.
	\newblock {\em J. Funct. Anal.}, 54(2):193--205, 1983.
	
	\bibitem{DPFlRo17}
	Giuseppe {Da Prato}, Franco {Flandoli}, and Michael {R{\"o}ckner}.
	\newblock {Continuity equation in LlogL for the 2D Euler equations under the
		enstrophy measure}.
	\newblock {\em arXiv e-prints}, page arXiv:1711.07759, Nov 2017.
	
	\bibitem{DPZa14}
	Giuseppe Da~Prato and Jerzy Zabczyk.
	\newblock {\em Stochastic equations in infinite dimensions}, volume 152 of {\em
		Encyclopedia of Mathematics and its Applications}.
	\newblock Cambridge University Press, Cambridge, second edition, 2014.
	
	\bibitem{DiSeSheWa15}
	Henk Dijkstra, Taylan Sengul, Jie Shen, and Shouhong Wang.
	\newblock Dynamic transitions of quasi-geostrophic channel flow.
	\newblock {\em SIAM J. Appl. Math.}, 75(5):2361--2378, 2015.
	
	\bibitem{Fl18}
	Franco Flandoli.
	\newblock Weak vorticity formulation of 2{D} {E}uler equations with white noise
	initial condition.
	\newblock {\em Comm. Partial Differential Equations}, 43(7):1102--1149, 2018.
	
	\bibitem{FlGrLu19}
	Franco {Flandoli}, Francesco {Grotto}, and Dejun {Luo}.
	\newblock {Fokker-Planck equation for dissipative 2D Euler equations with
		cylindrical noise}.
	\newblock {\em to appear on Stochastics and Dynamics}, page arXiv:1907.01994,
	Jul 2019.
	
	\bibitem{FlLu18}
	Franco {Flandoli} and Dejun {Luo}.
	\newblock {Kolmogorov equations associated to the stochastic 2D Euler
		equations}.
	\newblock {\em to appear on SIAM\ J. Math. Analysis}, page arXiv:1803.05654,
	Mar 2018.
	
	\bibitem{GrMu96}
	Federico Graef and Peter Müller.
	\newblock Uniqueness of solutions and conservation laws for the
	quasigeostrophic model.
	\newblock {\em Dynamics of Atmospheres and Oceans}, 25(2):109 -- 132, 1996.
	
	\bibitem{Gr19}
	Francesco {Grotto}.
	\newblock {Stationary Solutions of Damped Stochastic 2-dimensional Euler's
		Equation}.
	\newblock {\em arXiv e-prints}, page arXiv:1901.06744, Jan 2019.
	
	\bibitem{Ju63}
	V.~I. Judovi\v{c}.
	\newblock Non-stationary flows of an ideal incompressible fluid.
	\newblock {\em \v{Z}. Vy\v{c}isl. Mat. i Mat. Fiz.}, 3:1032--1066, 1963.
	
	\bibitem{MaBe02}
	Andrew~J. Majda and Andrea~L. Bertozzi.
	\newblock {\em Vorticity and incompressible flow}, volume~27 of {\em Cambridge
		Texts in Applied Mathematics}.
	\newblock Cambridge University Press, Cambridge, 2002.
	
	\bibitem{MaTiVE01}
	Andrew~J. Majda, Ilya Timofeyev, and Eric Vanden~Eijnden.
	\newblock A mathematical framework for stochastic climate models.
	\newblock {\em Comm. Pure Appl. Math.}, 54(8):891--974, 2001.
	
	\bibitem{MaWa06}
	Andrew~J. Majda and Xiaoming Wang.
	\newblock {\em Non-linear dynamics and statistical theories for basic
		geophysical flows}.
	\newblock Cambridge University Press, Cambridge, 2006.
	
	\bibitem{MaPu94}
	Carlo Marchioro and Mario Pulvirenti.
	\newblock {\em Mathematical theory of incompressible nonviscous fluids},
	volume~96 of {\em Applied Mathematical Sciences}.
	\newblock Springer-Verlag, New York, 1994.
	
	\bibitem{Os98}
	W.~F. Osgood.
	\newblock Beweis der {E}xistenz einer {L}\"{o}sung der {D}ifferentialgleichung
	{$\frac{{dy}}{{dx}} = f\left( {x,y} \right)$} ohne {H}inzunahme der
	{C}auchy-{L}ipschitz'schen {B}edingung.
	\newblock {\em Monatsh. Math. Phys.}, 9(1):331--345, 1898.
	
	\bibitem{Si87}
	Jacques Simon.
	\newblock Compact sets in the space {$L^p(0,T;B)$}.
	\newblock {\em Ann. Mat. Pura Appl. (4)}, 146:65--96, 1987.
	
\end{thebibliography}

\end{document}